\documentclass{amsart}
\usepackage{amsmath,amssymb,amsfonts}
   \newcommand{\mfa}{\mathfrak{a}}
   \newcommand{\mfb}{\mathfrak{b}}
   \newcommand{\mfc}{\mathfrak{c}}
   \newcommand{\mfd}{\mathfrak{d}}
\usepackage{enumitem}
   \newcommand{\lb}{label}
   \newcommand{\lm}{leftmargin}

\usepackage{euscript}
\usepackage{verbatim}
\usepackage{stmaryrd}
\newtheorem{theorem}{Theorem}[section]
\newtheorem{lemma}[theorem]{Lemma}

\newtheorem{example}[theorem]{Example}

\renewcommand{\ge}{\geqslant}
\newcommand{\eps}{\varepsilon}

\DeclareMathOperator{\Id}{Id}
\DeclareMathOperator{\e}{e}
\DeclareMathOperator{\Lip}{Lip}
\newcommand{\dint}{\displaystyle\int}
\newcommand{\dr}{\, dr}
\newcommand{\prts}[1]{\left(#1\right)}
\newcommand{\prtsr}[1]{\left[#1\right]}
\newcommand{\abs}[1]{\left|#1\right|}
\newcommand{\norms}[1]{\left\|#1\right\|_s}
\newcommand{\set}[1]{\left\{#1\right\}}
\newcommand{\setm}[1]{\setminus\set{#1}}

\usepackage{dsfont}
   \newcommand{\N}{\ensuremath{\mathds N}}
   \newcommand{\R}{\ensuremath{\mathds R}}
\def\cB{\EuScript{B}}
\def\cV{\EuScript{V}}
\def\cX{\EuScript{X}}
\renewcommand{\ge}{\geqslant}

\renewcommand{\le}{\leqslant}

\begin{document}
\title[Nonuniform dichotomic behavior: Lipschitz ...]%
   {Nonuniform dichotomic behavior: Lipschitz
   invariant manifolds for ODEs}
\author[Ant\'onio J. G. Bento]{Ant\'onio J. G. Bento}
\address{
   Ant\'onio J. G. Bento\\
   Departamento de Matem\'atica\\
   Universidade da Beira Interior\\
   6201-001 Covilh\~a\\
   Portugal}
\email{bento@ubi.pt}
\author{C\'esar M. Silva}
\address{
   C\'esar M. Silva\\
   Departamento de Matem\'atica\\
   Universidade da Beira Interior\\
   6201-001 Covilh\~a\\
   Portugal}
\email{csilva@ubi.pt}
\urladdr{www.mat.ubi.pt/~csilva}
\date{\today}
\begin{abstract}
   We obtain global and local theorems on the existence of invariant manifolds for perturbations of non autonomous linear differential equations assuming a very general form of dichotomic behavior for the linear equation. Besides some new situations that are far from the hyperbolic setting, our results include, and sometimes improve, some known stable manifold theorems.
\end{abstract}
\maketitle
\section{Introduction}
The study of invariant manifolds is an important and classical subject in the theory of dynamical systems and can be traced back to the work of Poincar\'e~\cite{Poincare-1886}, who developed new techniques to the study of differential equations, emphasizing the need to study globally the solutions and to use geometric tools in the description of the solutions. The stable manifold theorem, a fundamental tool in the dynamical systems approach to differential equations, goes back to the work of Hadamard and Perron. These authors established the two most used methods to obtain stable manifolds: Hadamard~\cite{Hadamard-1901} obtained stable and unstable manifolds of fixed points of diffeomorphisms using the graph transformation method that consists in constructing the manifolds as graphs over the linearized stable and unstable subspaces, while the method of Perron~\cite{Perron-JRAM-1929,Perron-MZ-1929,Perron-MZ-1930}, established for hyperbolic equilibrium points, uses the integral equation formulation of the differential equation to obtain the invariant manifolds as fixed point of an operator obtained from the integral equations. In this work, our approach is close to the one by Perron.

Since the pioneering work of Hadamard and Perron, successive extensions have been added to the theory of invariant manifolds. In particular, the hyperbolicity condition was relaxed by Pesin~\cite{Pesin-IANSSSR-1976,Pesin-IANSSSR-1977,Pesin-UMN-1977} that considered the weaker notion of nonuniform hyperbolicity and obtained stable and unstable manifolds. Another versions of the stable manifold theorem were established by Ruelle~\cite{Ruelle-AM-1982} in Hilbert spaces and by Ma\~n\'e~\cite{Mane-LNM-1983} in Banach spaces under some compactness and invertibility assumptions.

In order to study nonlinear perturbations of linear nonautonomous differential equations it is frequently assumed the existence of an exponential dichotomy for the linear equation. This concept was introduced by Perron~\cite{Perron-MZ-1929,Perron-MZ-1930} in the late 1920s and has undergone successive modifications and generalizations in the last decades. Namely two major paths towards generalization can be identified: growth rates that are not necessarily exponential and dichotomies that are nonuniform in the sense that the bounds for the growth depend both on the time elapsed and on the initial time. Uniform dichotomies with nonexponential growth rates were considered in the work of Pinto~\cite{Pinto-DEDS-1994} and Naulin and Pinto~\cite{Naulin-Pinto-NATMA-1994} where the authors study stability of ordinary differential linear equations possessing $(h,k)$-dichotomies. More recently, nonexponential growth rates, expressed as generalized exponential functions, can be found in the work of P\"otzsche (see~\cite{Potzsche-LNM-2010}). On the other hand, nonuniform exponential dichotomies can be found in Preda and Megan~\cite{Preda-Megan-BAusMS-1983}, Megan, Sasu and Sasu~\cite{Megan-Sasu-Sasu-IEOT-2002} and, in a different form, in the work of Barreira and Valls that obtained a large set of results for difference and differential equations assuming the existence of a nonuniform exponential dichotomy  (see~\cite{Barreira-Valls-LNM-2008}).

Naturally, one can relax the notion of dichotomy allowing nonexponential growth as well as nonuniform behavior. This approach was followed by the present authors that in~\cite{Bento-Silva-JFA-2009} obtained stable manifolds for nonautonomous nonlinear perturbations of nonautonomous linear difference equations, assuming the existence of a nonuniform polynomial dichotomy for the linear equation. In the context of differential equations, in~\cite{Barreira-Valls-NATMA-2009} it was shown by Barreira and Valls that the existence of a different type of nonuniform polynomial dichotomy follows from the nonvanishing of a certain polynomial Lyapunov exponents. Invariant manifolds for differential equations assuming the existence of these polynomial dichotomies were obtained by the present authors in~\cite{Bento-Silva-QJM-2012}. Other works considered general nonuniform behaviors for the dichotomy. Namely, in~\cite{Barreira-Valls-NATMA-2010-72-(2444-2456)} Barreira and Valls obtained global stable manifolds for perturbations of linear equations assuming that the linear equation admits a so-called $\rho$-dichotomy and in~\cite{Bento-Silva-JDDE} the present authors established the existence of local stable manifolds assuming that the linear equation admits the more general notion of $(\mu,\nu)$-dichotomy.

The several notions of dichotomy and the existence in the literature of related results obtained for each specific notion of dichotomy led us to  define a general framework that includes as particular cases the several definitions of nonuniform dichotomy and that still allows us to obtain general results. Namely, we will consider general dichotomic behavior that consists simply in assuming the existence of a splitting into invariant subspaces where the norms of the evolution map are bounded by some general functions that depend on the initial and final times.

In this paper we establish the existence of Lipschitz invariant manifolds for perturbations of non autonomous linear differential equations with the mentioned dichotomic behavior, obtaining an asymptotic behavior along the manifolds that is the same as the one assumed for the linear part in the corresponding subspaces.

Our approach has some advantages. Firstly, we are able to consider growth rates given by some arbitrary function and this makes our results independent of particular properties of functions such as the exponential functions or the polynomial functions (for instance we can consider non monotonic growth). Secondly, the assumptions in our theorems are given by relations between the growth rates in the dichotomy and Lipschitz constants associated to the perturbations and this allows us to determine easily if our theorem holds for some set of bounds given the Lipschitz constants of the perturbations or to find perturbations with small enough Lipschitz constants for the theorem to hold with prescribed growth rates. Finally, unlike some previous related results we require only invertibility of the linear evolution operator in some subspaces, which may allow us to apply our results to compact operators defined in infinite dimensional Banach spaces.

Another interesting aspect of our work is that we obtain a result on the existence of global invariant manifolds and use it to derive a corresponding result on the existence of local invariant manifolds. This procedure reveals
a link between the Lipschitz constants and the size of the manifolds and unifies the settings considered in previous works for local and global results.

A discrete time counterpart of our results can be found in~\cite{Bento-Silva-2012-arXiv1209.6589B}.
\section{Notation and preliminaries}
Let $(\R_0^+)^2_\ge=\set{\prts{t,s} \in (\R_0^+)^2 \colon t \ge s}$ and $(\R_0^+)^2_> = \set{(t,s) \in (\R_0^+)^2 \colon t > s}$. Let $X$ be a Banach space and consider a continuous map $A:\R_0^+ \to B(X)$, where $B(X)$ denotes the set of bounded linear operators in $X$. Consider also the linear initial value problem
\begin{equation} \label{eq:lin:dif}
   v' = A(t) v, \ v(s)=v_s
\end{equation}
and assume that the solutions of~\eqref{eq:lin:dif} are global in the future. For each $(t,s) \in (\R_0^+)^2_\ge$, denote by $T_{t,s}$ the linear evolution operator associated to equation~\eqref{eq:lin:dif}.

We say that equation~\eqref{eq:lin:dif} admits an \emph{invariant splitting} if there exist bounded projections $P_s$, $s \in \R_0^+$, such that, for every $(t,s) \in (\R_0^+)^2_\ge$ we have
\begin{enumerate}[\lb=$($S$\arabic*)$,\lm=13mm]
   \item $T_{t,s}P_s=P_tT_{t,s}$;
   \item $T_{t,s}(\ker P_s) = \ker P_t$;
   \item $T_{t,s}|_{\ker P_s}: \ker P_s \to \ker P_t$ is invertible with bounded inverse.
\end{enumerate}
We define, for each $t \in \R_0^+$, the complementary projection $Q_t = \Id - P_t$ and the linear subspaces $E_t=P_t(X)$ and $F_t= \ker P_t = Q_t(X)$. As usual, we identify the vector spaces $E_t \times F_t$ and $E_t \oplus F_t$ as the same vector space.

Given functions $a,b:(\R_0^+)^2_\ge \to \R^+$, we say that equation~\eqref{eq:lin:dif} admits a \emph{general dichotomy with bounds $a$ and $b$} if it admits an invariant splitting such that
\begin{enumerate}[\lb=$($D$\arabic*)$,\lm=13mm]
   \item $\|T_{t,s}P_s\| \le a(t,s)$; \label{eq:splitbounds1}
   \item $\|(T_{t,s}|_{F_s})^{-1} Q_t\| \le b(t,s)$ where $Q_t=\Id-P_t$. \label{eq:splitbounds2}
\end{enumerate}
\section{Existence of Lipschitz manifolds}
In this section we are going to state our results on the existence of Lipschitz invariant manifolds of the initial value problem
\begin{equation}\label{eq:ivp-nonli}
	v' = A(t)v + f(t,v), \ v(s)=v_s
\end{equation}
where $f: \R_0^+ \times X \to X$ is a continuous function such that
\begin{equation}\label{cond-f-0}
   f(t,0)=0 \text{ for every $t \in \R_0^+$}
\end{equation}
and, for each $t \in \R_0^+$, the function $f_t:X \to X$ given by $f_t(x)=f(t,x)$ is a Lipschitz function. Denoting by $\Lip(f_t)$ the Lipschitz constant of $f_t$, it is clear that
\begin{equation}\label{cond-f-1}
   \|f(t,x) - f(t,y)\| \le \Lip(f_t) \|x-y\|
\end{equation}
for every $t \in \R_0^+$ and every $x,y \in X$ and making $y =0$ in~\eqref{cond-f-1}, by~\eqref{cond-f-0}, we have
\begin{equation}\label{cond-f-2}
   \|f(t,x) \| \le \Lip(f_t) \|x\|
\end{equation}
for every $t \in \R_0^+$ and every $x \in X$.

Note that condition~\eqref{cond-f-0} implies that $v(t) \equiv 0$ is the solution of~\eqref{eq:ivp-nonli} with $v_s =0$. This is not a serious restriction because if $v_0(t)$ is a nonzero solution of
   $$ v' = A(t) v + f(t,v),$$
then the change of variables $(t,u)=(t,v-v_0(t))$ transforms the previous problem in the problem
   $$ v' = A(t) v + g(t,v)$$
where $g(t,v) = f(t,v_0(t)+v) - f(t,v_0(t))$ and it is straighforward to see that, for every $t \in \R_0^+$, $g(t,\cdot)$ is a Lipschitz function with Lipschitz constant $\Lip(f_t)$ and $g(t,0)=0$.

Writing the unique solution of \eqref{eq:ivp-nonli} in the form
   $$ (x(t,s,v_s),y(t,s,v_s)) \in E_t \times F_t$$
where $v_s = (\xi, \eta) \in E_s \times F_s$, problem~\eqref{eq:ivp-nonli} is equivalent to the following problem
\begin{align}
   & x(t) = T_{t,s} \xi + \int_s^t T_{t,r} P_r f(r,x(r),y(r)) \dr,
      \label{eq:split-1a}\\
   & y(t) = T_{t,s} \eta + \int_s^t T_{t,r} Q_r f(r,x(r),y(r)) \dr.
   \label{eq:split-1b}
\end{align}

For each $\tau \ge 0$ we define the semiflow
\begin{equation} \label{def:Psi}
   \Psi_\tau(s,v_s) = \prts{s+\tau, x(s+\tau,s,v_s), y(s+\tau, s, v_s)}.
\end{equation}

Define
   $$ G = \bigcup_{t \ge 0} \set{t} \times E_t$$
and denote by $\cX$ the space of functions $\phi:G \to X$ such that
\begin{align}
   & \phi(t,0)=0, \label{cond-phi-0} \\
   & \phi(t,\xi) \in F_t \label{cond-phi-0a}\\
   & \| \phi(t,\xi) - \phi(t,\bar\xi) \| \le \| \xi - \bar\xi\|,
   \label{cond-phi-1}
\end{align}
for every $(t,\xi), (t,\bar\xi) \in G$. Note that making $\bar\xi = 0$ in~\eqref{cond-phi-1} we have
\begin{equation} \label{cond-phi-1a}
   \|\phi(t,\xi)\| \le \|\xi\|
\end{equation}
for every $(t,\xi) \in G$.

For every $\phi \in \cX$ we define the graph
\begin{equation} \label{def:V_phi}
   \cV_\phi
   = \set{\prts{s,\xi, \phi(s,\xi)} \colon (s,\xi) \in G},
\end{equation}
that we call \emph{global Lipschitz invariant manifold}.

We now state the result on the existence of global invariant manifolds.

\begin{theorem} \label{thm:global}
   Given a Banach space $X$, suppose that equation~\eqref{eq:lin:dif} admits a general dichotomy with bounds $a,b:(\R_0^+)^2_\ge \to \R^+$. Let $f: \R_0^+ \times X \to X$ be a continuous function such that $f_t$ is Lipschitz for each $t \in \R_0^+$ and~\eqref{cond-f-0} holds. Assume that
   \begin{equation} \label{eq:CondicaoTeo}
      \lim_{t \to +\infty} a(t,s) b(t,s) = 0
   \end{equation}
   for every $s \in \R_0^+$,
   \begin{equation}\label{eq:tau}
      \alpha =
      \sup_{(t,s) \in (\R_0^+)^2_>} \
      \dfrac{1}{a(t,s)} \dint_s^t a(t,r) a(r,s) \Lip(f_r) \dr
      < + \infty
   \end{equation}
   and
   \begin{equation}\label{eq:lbd}
      \beta =
      \sup_{s \in \R_0^+} \dint_s^{+\infty} b(r,s) a(r,s) \Lip(f_r) \dr
      < +\infty.
   \end{equation}
   If
   \begin{equation}\label{ine:alpha_beta}
      2 \alpha + \max\set{2 \beta, \sqrt{\beta}} < 1,
   \end{equation}
   then there is a unique $\phi \in \cX$ such that
   \begin{equation} \label{eq:invariance}
      \Psi_\tau(\cV_\phi) \subseteq \cV_\phi
   \end{equation}
   for every $\tau \ge 0$, where $\Psi_\tau$ is given by~\eqref{def:Psi} and
   $\cV_\phi$ is given by~\eqref{def:V_phi}. Moreover,
   \begin{equation} \label{eq:bound-thm1}
      \| \Psi_{t-s}(s,\xi,\phi(s,\xi))- \Psi_{t-s}(s,\bar\xi,\phi(s,\bar\xi))\|
      \le \frac{2}{1-2\alpha} a(t,s) \|\xi - \bar\xi\|
   \end{equation}
   for every $(t,s) \in (\R_0^+)^2_\ge$ and every $\xi, \bar \xi \in E_s$.
\end{theorem}

The next theorem is a local version of Theorem~\ref{thm:global}. Let $B(r)$ denote the open ball of radius $r$ in $X$. Given a function $R:\R_0^+ \to \R^+$ define
\begin{equation}\label{eq:manifold-local}
   \cV^*_{\phi,R}
   =\set{\prts{s,\xi,\phi(s,\xi)} \in \cV_\phi \colon
      (s,\xi) \in \{s\} \times  B(R(s)) }.
\end{equation}
We have the following result.

\begin{theorem} \label{thm:local}
   Given a Banach space $X$, suppose that equation~\eqref{eq:lin:dif} admits a splitting with bounds $a$ and $b$. Let $f \colon \R_0^+ \times X \to X$ be a continuous function satisfying \eqref{cond-f-0} and such that, for each $t \in \R_0^+$, $f_t : X \to X$ is a Lipschitz function in $B(R(t))$, where $R \colon \R_0^+ \to \R^+$. Assume that
   \begin{equation} \label{eq:CondicaoTeo-local}
      \lim_{t \to +\infty} a(t,s) b(t,s) = 0
   \end{equation}
   for every $s \in \R_0^+$,
   \begin{equation}\label{eq:tau-local}
      \alpha =
      \sup_{(t,s) \in (\R_0^+)^2_>}
      \dfrac{1}{a(t,s)} \int_s^t a(t,r) a(r,s)
      \Lip(f_r|_{B(R(r))}) \dr < + \infty
   \end{equation}
   and
   \begin{equation}\label{eq:lbd-local}
      \beta =
      \sup_{s \in \R_0^+} \dint_s^{+\infty} b(r,s) a(r,s)
      \Lip(f_r|_{B(R(r))}) \dr < + \infty.
   \end{equation}
   If, for each $s \in \R_0^+$,
   \begin{equation}\label{ine:sn-local}
      S(s)= \max\left\{ 1, \, \frac{2}{1-4\alpha} \, \sup_{t \ge s} \frac{a(t,s) R(s)}{R(t)} \right\}< +\infty
   \end{equation}
   and
   \begin{equation}\label{ine:alpha_beta-local}
      4 \alpha + \max\set{4 \beta, \sqrt{2\beta}} < 1,
   \end{equation}
   then there is $\phi \in \cX$ such that
   \begin{equation} \label{thm:local:invar-local}
      \Psi_\tau(\cV^*_{\phi,R/(2S)}) \subseteq \cV^*_{\phi,R}
   \end{equation}
   for every $\tau \ge 0$, where $\Psi_\tau$ is given by~\eqref{def:Psi}, $R/(2S)$ denotes the function given by $s \mapsto R(s)/(2S(s))$ and $\cV^*_{\phi,R/(2S)}$ and $\cV^*_{\phi,R}$ are given by~\eqref{eq:manifold-local}. Furthermore, we have
   \begin{equation}\label{thm:ineq:norm:F_mn(xi...)-F_mn(barxi...)-local}
      \|\Psi_{t-s}(\xi,\phi(s,\xi)) -\Psi_{t-s}(\bar \xi,\phi(s,\bar \xi)) \|
      \le \dfrac{2}{1-4\alpha} \, a(t,s)\, \|\xi-\bar \xi\|.
   \end{equation}
   for every $\prts{t,s} \in (\R_0^+)^2_\ge$ and every $\xi, \bar\xi \in B(R(s))$.
\end{theorem}
\section{Examples}
In this section we will give examples that illustrate our results and show that our general setting contains several known results in the literature. Firstly, we will consider bounds $a,b$ of the form
\begin{equation}\label{cond:mfa...1}
   a(t,s)= \dfrac{\mfa(s)}{\mfa(t)} \mfc(s)
   \quad \text{ and } \quad
   b(t,s)=\dfrac{\mfb(s)}{\mfb(t)} \mfd(t)
\end{equation}
where $\mfa, \mfb, \mfc, \mfd:\R_0^+ \to \R^+$ are some positive functions and
\begin{equation}\label{cond:mfa...2}
   \mfc(t) , \mfd(t) \ge 1 \text{ for all } t \in \R^+_0.
\end{equation}

Our fist example shows that there is always a differential equation that has a generalized dichotomy with bounds $a$ and $b$ of the form~\eqref{cond:mfa...1} with $\mfa,\mfb,\mfc,\mfd$ differentiable and satisfying~\eqref{cond:mfa...2}.

\begin{example}
   Let $\mfa,\mfb,\mfc,\mfd \colon \R^+_0 \to \R^+$ be positive differentiable functions satisfying~\eqref{cond:mfa...2}. The differential equation in $\R^2$ given by
   \begin{equation} \label{eq:example}
      \begin{cases}
         u' = \prts{- \dfrac{\mfa'(t)}{\mfa(t)}
            + \dfrac{\mfc'(t)}{\mfc(t)} \, \dfrac{\cos t -1}{2}
            - \log \mfc(t) \dfrac{\sin t}{2}} u\\[3mm]
         v' = \prts{\dfrac{\mfb'(t)}{\mfb(t)}
            + \dfrac{\mfd'(t)}{\mfd(t)} \, \dfrac{\cos t -1}{2}
            - \log \mfd(t) \dfrac{\sin t}{2}} v
      \end{cases}
   \end{equation}
   has the following evolution operator
      $$ T(t,s)(u,v) = (U(t,s)u, V(t,s)v),$$
   where
   \begin{align*}
      & U(t,s)
         = \dfrac{\mfa(s)}{\mfa(t)}
            \dfrac{\mfc(t)^{(\cos t-1)/2}}{\mfc(s)^{(\cos s-1)/2}},\\
      & V(t,s) = \dfrac{\mfb(t)}{\mfb(s)}
            \dfrac{\mfd(t)^{(\cos t-1)/2}}{\mfd(s)^{(\cos s-1)/2}}.
   \end{align*}
   Using the projections $P(t) \colon \R^2 \to \R^2$ defined by $P(t)(u,v)
   =(u,0)$ we have
   \begin{align*}
      & \|T_{t,s}P_s\|
         = \abs{U(t,s)}
         \le \dfrac{\mfa(s)}{\mfa(t)} \mfc(s)\\
      & \|\prts{T_{t,s}|_{F_s}}^{-1} Q_t\|
         = \abs{V(t,s)^{-1}}
         \le \dfrac{\mfb(s)}{\mfb(t)} \mfd(t)
   \end{align*}
   and thus~\eqref{eq:example} admits a general dichotomy with bounds $a$ and $b$ of the form~\eqref{cond:mfa...1}. Moreover,
   if $t = 2 k \pi$ and $s = (2k -1) \pi$, $k \in \N$, then
      $$ \|T_{t,s}P_s\| = \dfrac{\mfa(s)}{\mfa(t)} \mfc(s)$$
   and if $t = (2 k-1) \pi$ and $s = 2k \pi$, $k \in \N$, then
      $$ \|\prts{T_{t,s}|_{F_s}}^{-1} Q_t\| = \dfrac{\mfb(s)}{\mfb(t)} \mfd(t).$$
\end{example}

In the next example we show that we can obtain invariant manifolds for generalized dichotomies with bounds of the form~\eqref{cond:mfa...1} satisfying~\eqref{cond:mfa...2}.

\begin{example}
   Let $a$ and $b$ be bounds of the form~\eqref{cond:mfa...1} satisfying~\eqref{cond:mfa...2}. In this particular case, conditions~\eqref{eq:CondicaoTeo}, \eqref{eq:tau} and \eqref{eq:lbd} correspond respectively to the conditions
	\begin{equation} \label{eq:condddd}
	  \lim_{t \to +\infty} \dfrac{\mfd(t)}{\mfa(t) \mfb(t)} = 0,
	\end{equation}
	  $$\alpha=\dint_0^{+\infty} \mfc(r) \Lip(f_r) \dr < +\infty$$
   and
    	$$\beta=\sup_{s \in \R_0^+} \mfa(s) \mfb(s) \mfc(s) \dint_s^{+\infty}
 	\dfrac{\mfd(r)}{\mfa(r) \mfb(r)} \Lip(f_r) \dr < +\infty.$$
   Thus, if the numbers $\Lip(f_r)$ are small enough so the last three conditions and~\eqref{ine:alpha_beta} hold, we obtain an invariant manifold $\cV_\phi$ given by~\eqref{def:V_phi} where
      $$ \| \Psi_{t-s}(s,\xi,\phi(s,\xi))- \Psi_{t-s}(s,\bar\xi,\phi(s,\bar\xi))\| \le \frac{2}{1-2\alpha} \dfrac{\mfa(s)}{\mfa(t)} \mfc(s) \|\xi - \bar\xi\|, $$
   for every $\prts{t,s} \in (\R_0^+)_\ge^2$ and every $\xi, \bar\xi \in E_s$.

   It is easy to see that if $\mfa, \mfb, \mfc$ are non decreasing and
      $$ \Lip(f_r) \le \frac{\lambda(r)}{\mfc(r)\mfd(r)} \quad \text{with} \quad \dint_0^{+\infty} \lambda(r) \dr < \frac{1}{4},$$
   conditions~\eqref{eq:tau}, \eqref{eq:lbd} and~\eqref{ine:alpha_beta} are verified and thus, provided that~\eqref{eq:condddd} holds, we always have perturbations with small enough non-zero Lipschitz constants such that the perturbed equations have invariant manifolds with the behavior given in our theorem.

   In particular, setting
      $$ \mfa(r)=\e^{-ar}, \ \ \
         \mfb(r)=\e^{br} \ \ \ \text{ and } \ \ \
         \mfc(r)=\mfd(r)=D \e^{\eps r},$$
   for some constants $D \ge 1$, $a< 0 \le b$ and $\eps>0$, we get
   \begin{equation}\label{a(t,s)=exp...}
      a(t,s)=D\e^{a(t-s)+\eps s}
      \quad \text{ and } \quad
      b(t,s)=D\e^{b(t-s)+\eps t},
   \end{equation}
   and if
      $$ \Lip(f_r) \le \delta \e^{-2\eps r},$$
   choosing $\delta > 0$ small enough, we obtain a Lipschitz version of Theorem 4 in~\cite{Barreira-Valls-DCDS-A-2008-21-4}. Note that, for these dichotomies, condition~\eqref{eq:condddd} is equivalent to condition $a+\eps<b$, already present in the referred paper.

   Another particular case can be obtained by setting
      $$ \mfa(r)=(r+1)^{-a}, \ \ \
         \mfb(r)=(r+1)^b \ \ \ \text{ and } \ \ \
         \mfc(r)=\mfd(r)=D(r+1)^\eps,$$
   for some constants $D \ge 1$, $a<0 \le b$ and $\eps>0$, yelding
      $$ a(t,s)=D\left( \dfrac{t+1}{s+1}\right)^a (s+1)^\eps
         \quad \text{ and } \quad
         b(t,s)=D \left( \dfrac{t+1}{s+1}\right)^{-b} (t+1)^\eps.$$
   Since condition~\eqref{eq:condddd} is also equivalent to $a+\eps<b$, assuming this condition and setting
      $$ \Lip(f_r) \le \delta (r+1)^{-2\eps-1}$$
   with $\delta > 0$ sufficiently small, we obtain a Lipschitz version of Theorem 3.1 in~\cite{Bento-Silva-QJM-2012}.

   Setting
      $$ \mfa(r)=\e^{-a \rho(r)}, \ \ \
         \mfb(r)=\e^{b \rho(r)}\ \ \ \text{ and } \ \ \
         \mfc(r)=\mfd(r)=D \e^{\eps \rho(r)},$$
   and
      $$ \Lip(f_r) \le \delta \rho'(r) \e^{-2\eps \rho(r)},$$
   where $D \ge 1$, $a<0 \le b$ and $\eps>0$, $\rho \colon \R_0^+ \to \R_0^+$ is a nondecreasing differentiable function and $\delta > 0$ sufficiently small, we obtain a Lipschitz version of Theorem 1 in~\cite{Barreira-Valls-NATMA-2010-72-(2444-2456)}.
\end{example}

In the next example we consider dichotomies with bounds that are not of the form~\eqref{cond:mfa...1}.

\begin{example}
   For each $(t,s) \in (\R_0^+)^2_\ge$, set
      $$ a(t,s) = D(t-s+1)^a (s+1)^\eps \quad \text{and} \quad
         b(t,s)=D(t-s+1)^{-b} (t+1)^\eps,$$
   for some constants $a<0\le b$ and $\eps>0$. Set also
      $$ \Lip(f_r)\le \delta (r+1)^{-2\eps-1}.$$
   In this case, conditions~\eqref{eq:tau} and \eqref{eq:lbd} are satisfied, condition~\eqref{eq:CondicaoTeo} corresponds to $a+\eps<b$ and condition~\eqref{ine:alpha_beta} is satisfied if we consider a small enough $\delta>0$. Thus, our theorem allows us to obtain an invariant stable manifold $\cV_\phi$ given by~\eqref{def:V_phi} where the decay is given by
      $$ \| \Psi_{t-s}(s,\xi,\phi(s,\xi))
            - \Psi_{t-s}(s,\bar\xi,\phi(s,\bar\xi))\|
         \le \frac{2D}{1-2\alpha} (t-s+1)^a (s+1)^\eps \|\xi - \bar\xi\|,$$
   for every $(t,s) \in (\R_0^+)^2_\ge$ and every $\xi, \bar\xi \in E_s$.
\end{example}

In the next example we will see that replacing in~\eqref{a(t,s)=exp...} the bound $b$ we obtain invariant stable manifolds for any constants $a < 0$ and $\eps > 0$.

\begin{example}
   For each $(t,s) \in (\R_0^+)^2_\ge$, set
      $$ a(t,s) = D \e^{a(t-s)+\eps s} \quad \text{and} \quad
         b(t,s)=D \prts{\frac{t+1}{s+1}}^{-b} (t+1)^\eps,$$
   for some constants $a<0 \le b$ and $\eps>0$. Set also
      $$ \Lip(f_r) \le \delta \e^{-2\eps r}.$$
   In this case, all conditions of our theorem are satisfied provided that we consider a small enough $\delta>0$. Thus, our theorem allows us to obtain an invariant manifold $\cV_\phi$ given by~\eqref{def:V_phi} where the decay is given by
      $$ \| \Psi_{t-s}(s,\xi,\phi(s,\xi))
            - \Psi_{t-s}(s,\bar\xi,\phi(s,\bar\xi))\|
         \le \frac{2D}{1-2\alpha} \e^{a(t-s)+\eps s}\|\xi - \bar\xi\|,$$
   for every $(t,s) \in (\R_0^+)^2_\ge$ and every $\xi, \bar\xi \in E_s$.
\end{example}

The next example shows that we can still obtain some information about the dynamics in situations that are far from being hyperbolic in any reasonable sense.

\begin{example}
   For each $(t,s) \in (\R_0^+)^2_\ge$, set
      $$ a(t,s) = L \quad \text{and} \quad
         b(t,s)=D \e^{a(t-s)+\eps t},$$
   for some constants $L \ge 1$, $a<0$ and $\eps>0$. Set also
      $$ \Lip(f_r) \le \delta \e^{-\eps r}.$$
   Once again all conditions of our theorem are satisfied provided that we consider a small enough $\delta>0$. Thus, our theorem allows us to obtain a sequence of invariant manifolds $\cV_\phi$ given by~\eqref{def:V_phi} where we have
      $$ \| \Psi_{t-s}(s,\xi,\phi(s,\xi))
            - \Psi_{t-s}(s,\bar\xi,\phi(s,\bar\xi))\|
         \le \frac{2L}{1-2\alpha} \|\xi - \bar\xi\|,$$
   for every $(t,s) \in (\R_0^+)^2_\ge$ and every $\xi, \bar\xi \in E_s$. That is, we obtain an upper bound for the distance of the iterates of any two points in the manifolds.
\end{example}

Next we use Theorem~\ref{thm:local} to obtain local versions of the previous examples.

\begin{example}
   Let $a$ and $b$ be bounds of the form~\eqref{cond:mfa...1} satisfying~\eqref{cond:mfa...2} and assume that, for each $r \in \R_0^+$ and every $u, v \in X$, we have
   	$$ \|f(r,u)-f(r,v)\| \le c \|u-v\| (\|u\|+\|v\|)^q,$$
   for some constants $c>0$ and $q > 0$. It is immediate that $f_r|_{B(R(r))}$ is Lipschitz with Lipschitz constant less or equal than $2^q c R(r)^q$.
   Thus, conditions~\eqref{eq:CondicaoTeo-local}, \eqref{eq:tau-local}, \eqref{eq:lbd-local} and~\eqref{ine:sn-local} correspond respectively to the conditions
   \begin{equation*}
   	\lim_{t \to +\infty} \dfrac{\mfd(t)}{\mfa(t) \mfb(t)} = 0,
   \end{equation*}
   \begin{equation*}
   	\alpha = 2^q c \dint_0^{+\infty} \mfc(r) R(r)^q \dr < +\infty,
   \end{equation*}
   \begin{equation*}
   	\beta = 2^q c \sup_{s \in \R_0^+} \mfa(s) \mfb(s) \mfc(s)
      \dint_s^{+\infty} \dfrac{\mfd(r) R(r)^q}{\mfa(r) \mfb(r)} \dr
      < +\infty
   \end{equation*}
   and
   \begin{equation*}
   	\mfa(s) \mfc(s) R(s) \sup_{t \ge s} \mfa(t)^{-1} R(t)^{-1} < +\infty.
   \end{equation*}
   Thus, if the radius of the balls $B(R(r))$ is small enough so that~\eqref{ine:alpha_beta-local} holds, we obtain an invariant manifold given by~\eqref{eq:manifold-local} where the decay is given by
   $$ \| \Psi_{t-s}(s,\xi,\phi(s,\xi))
         - \Psi_{t-s}(s,\bar\xi,\phi(s,\bar\xi))\|
      \le \frac{2}{1-2\alpha} \dfrac{\mfa(s)}{\mfa(t)} \mfc(s)
         \|\xi - \bar\xi\|,$$
   for every $(t,s) \in (\R_0^+)^2_\ge$ and every $\xi, \bar\xi \in B(R(s))$.

   Note that, if $\mfa$, $\mfb$ and $\mfc$ are non decreasing functions,
      $$\sup_{s \in \R_0^+} \mfc(s)
         \int_s^{+\infty} \dfrac{\mfd(r)}{\mfa(r)^q} \dr
         < +\infty$$
   and
      $$ \int_0^{+\infty} \dfrac{\mfc(r)}{\mfa(r)^q} \dr < +\infty,$$
   putting $R(r) = \delta/\mfa(r)$ and choosing $\delta > 0$ sufficiently small we are in the conditions of Theorem~\ref{thm:local} and therefore we obtain a local invariant manifold theorem.

   As a particular case, given $a < 0 \le b$ and $\eps \ge 0$, we can put
      $$ \mfa(t)= \e^{- a t}, \ \
         \mfb(t)= \e^{b t}, \ \ \text{and} \ \
         \mfc(t)= \mfd(t)=D \e^{\eps t},$$
   for each $t \in \R_0^+$. We can also set
      $$ R(r)=\delta \e^{-\beta r}$$
   for each $r \in \R_0^+$, where $\beta$ and $\delta$ are positive numbers. In this setting condition~\eqref{eq:CondicaoTeo-local} is equivalent to $a+\eps<b$, condition~\eqref{eq:tau-local} is equivalent to $\eps-\beta q<0$, condition~\eqref{eq:lbd-local} is equivalent to $2\eps-\beta q \le 0$ and condition~\eqref{ine:sn-local} is equivalent to $a+\beta \le 0$. In this setting, choosing $\delta>0$ sufficiently small, if $a + \eps < b$ and $a + 2\eps/q \le 0$, we obtain a local stable manifold for every positive number $\beta \in [2\eps/q,-a]$. This improves Theorem 3 in~\cite{Barreira-Valls-JDE-2006-221} (see also Theorem 4.1 in~\cite{Barreira-Valls-LNM-2008}) since in that paper $\beta = \eps(1+2/q)$.

   Another particular case is obtained putting
      $$ \mfa(t)= (t+1)^{- a}, \ \
         \mfb(t)= (t+1)^b, \ \ \text{and} \ \
         \mfc(t)= \mfd(t)=D (t+1)^\eps,$$
   and
      $$ R(r)= \delta (r+1)^{-\beta}$$
   where $a < 0 \le b$, $\eps > 0$ and $\beta > 0$. In these conditions,~\eqref{eq:CondicaoTeo-local} is equivalent to $a+\eps<b$, condition~\eqref{eq:tau-local} is equivalent to $\eps+1-\beta q <0$, condition~\eqref{eq:lbd-local} is equivalent to $2\eps + 1 -\beta q \le 0$ and condition~\eqref{ine:sn-local} is equivalent to $a+\beta \le 0$. Therefore if $a + \eps < b$ and $a + (2\eps + 1)/q \le 0$, choosing $\delta > 0$ sufficiently small, we obtain a local stable manifold theorem for every $\beta \in [(2\eps+1)/q,-a]$ and this improves Theorem 4.1 in~\cite{Bento-Silva-QJM-2012} since in that paper it was necessary to have $\beta = \eps(1+2/q) + 1/q$.
\end{example}

\begin{example}
   For each $(t,s) \in (\R_0^+)^2_\ge$, set
   \begin{equation*}
      a(t,s)= D \prts{\dfrac{\mu(t)}{\mu(s)}}^a \nu(s)^\eps
      \quad \text{and} \quad
      b(t,s)=D \prts{\dfrac{\mu(t)}{\mu(s)}}^{-b} \nu(t)^\eps
   \end{equation*}
   where $\mu,\nu:\R_0^+ \to \R^+$ are growth rates, that is these functions are non decreasing, converge to $+\infty$ and $\mu(0)=\nu(0)=1$. Assume also, for each $r \in \R_0^+$, that $R(r)=\delta R_0(r)$ where $\delta$ is a positive number and that, for every $ r \in \R_0^+$ and every $u,v \in X$, we have
   	$$ \|f(r,u)-f(r,v)\| \le c \|u-v\| (\|u\|+\|v\|)^q,$$
   for some constants $c>0$ and $q > 0$. In this case, conditions~\eqref{eq:CondicaoTeo-local}, \eqref{eq:tau-local}, \eqref{eq:lbd-local} and~\eqref{ine:sn-local} correspond respectively to the conditions
   \begin{equation} \label{eq:cond1a-local}
   	\lim_{t \to +\infty} \mu(t)^{a-b} \nu(t)^\eps = 0,
   \end{equation}
   \begin{equation} \label{eq:cond2a-local}
   	\alpha = 2^q c D \delta^q \dint_0^{+\infty} \nu(r)^\eps R_0(r)^q \dr
      < +\infty,
   \end{equation}
   \begin{equation} \label{eq:cond3a-local}
      \beta = 2^q c D^2 \delta^q \sup_{s \in \N} \mu(s)^{b-a} \nu(s)^\eps \dint_s^{+\infty}
    	\mu(r)^{a-b} \nu(r)^\eps R_0(r)^q \dr < +\infty
    \end{equation}
   and
   \begin{equation} \label{eq:cond4a-local}
   	\frac{R_0(s) \nu(s)^\eps}{\mu(s)^a} \sup_{t \ge s}
   	\frac{\mu(t)^a}{R_0(t)} < +\infty.
   \end{equation}

   Thus, if the last four conditions are satisfied and $\delta >0$ is small enough so that~\eqref{ine:alpha_beta-local} holds, we obtain an invariant manifold $\cV^*_{\phi,R}$ given by~\eqref{eq:manifold-local} where
      $$ \| \Psi_{t-s}(s,\xi,\phi(s,\xi))
            - \Psi_{t-s}(s,\bar\xi,\phi(s,\bar\xi))\|
         \le \frac{2}{1-2\alpha} D \prts{\dfrac{\mu(t)}{\mu(s)}}^a \nu(s)^\eps
            \|\xi - \bar\xi\|,$$
   for every $(t,s) \in (\R_0^+)^2_\ge$ and every $\xi, \bar\xi \in B(R(s))$.

   Letting now $R_0(r)=\mu(r)^a$ and supposing that
      $$ \int_0^{+\infty} \mu(r)^{aq} \nu(r)^\eps \dr < + \infty$$
   it is clear that~\eqref{eq:cond2a-local},~\eqref{eq:cond3a-local} and~\eqref{eq:cond4a-local} are satisfied. Thus, assuming that~\eqref{eq:cond1a-local} holds and choosing $\delta > 0$ sufficiently small, we obtain a local stable manifold theorem that improves Theorem 2.1 in~\cite{Bento-Silva-JDDE}.
\end{example}

\begin{example}
   Given $a < 0 \le b$, $\eps \ge 0$ and $D \ge 1$, let
      $$ a(t,s) = D \e^{a(\rho(t)-\rho(s))+\eps \rho(s)}
         \quad \text{ and } \quad
         b(t,s) = D \e^{-b(\rho(t)-\rho(s))+\eps \rho(t)},$$
   where $\rho : \R_0^+ \to \R_0^+$ is an increasing $C^1$ function such that
      $$ \lim_{t \to +\infty} \dfrac{\log t}{\rho(t)} = 0.$$
   Assume that
      $$ R(r) = \delta \rho'(t)^{1/q} \e^{-\beta \rho(r)}$$
   where $\delta$ and $\beta$ are positive numbers. Then~\eqref{eq:CondicaoTeo-local} is equivalent to $a+\eps<b$, condition~\eqref{eq:tau-local} is equivalent to $\eps-\beta q <0$, condition~\eqref{eq:lbd-local} is equivalent to $2\eps -\beta q \le 0$ and $a+\beta < 0$ implies condition~\eqref{ine:sn-local}. Hence, if $a+\eps < b$ and $2\eps/q + a < 0$, choosing $\beta \in [2\eps/q,-a[$ and $\delta$ sufficiently small we get a local stable manifold theorem that improves Theorem 2 in~\cite{Barreira-Valls-JFA-2009-257-(1018-1029)}.
\end{example}

\begin{example}
   For each $(t,s) \in (\R_0^+)^2_\ge$, set
      $$ a(t,s) = (t-s+1)^a (s+1)^\eps
         \quad \text{and} \quad
         b(t,s)=(t-s+1)^{-b} (t+1)^\eps,$$
   for some constants $a < 0 \le b$ and $\eps>0$. Assume further that
   for every $r \in \R_0^+$ and every $u,v \in X$, we have
   	$$\|f(r,u)-f(r,v)\| \le c \|u-v\| (\|u\|+\|v\|)^q,$$
   for some $c>0$ and $q>0$ and that $R(r)=\delta (r+1)^{-\beta}$. In this case, conditions~\eqref{eq:tau-local} and \eqref{eq:lbd-local} are satisfied provided that $\beta \ge (2\eps+1)/q$, condition~\eqref{ine:sn-local} is satisfied if $a + \beta \le 0$ and condition~\eqref{eq:CondicaoTeo-local} corresponds to $a+\eps<b$. Therefore, if $a+\eps < b$, $(2\eps+1)/q \le -a$ and $\beta \in [(2\eps+1)/q,-a]$, considering a small enough $\delta>0$, we have a local stable manifold theorem.
\end{example}

\section{Proof of Theorem~\ref{thm:global}}
Next we will prove Theorem~\ref{thm:global}. Given $s \in \R_0^+$ and $v_s=(\xi,\eta)\in E_s \times F_s$, by~\eqref{eq:split-1a} and~\eqref{eq:split-1b}, considering the invariance in~\eqref{eq:invariance}, we conclude that, for each $t \ge s$, we must have
\begin{align}
   x(t,\xi) & = T_{t,s} \xi + \dint_s^t T_{t,r} P_r f(r,x(r,\xi),\phi(r,x(r,\xi))) \dr,
      \label{eq:dyn-split2a}\\
   \phi(t,x(t,\xi)) &= T_{t,s} \eta +
   \dint_s^t T_{t,r} Q_r f(r,x(r,\xi),\phi(r,x(r,\xi))) \dr,
      \label{eq:dyn-split2b}
\end{align}
for some $\phi \in \cX$.

To prove that equations~\eqref{eq:dyn-split2a} and~\eqref{eq:dyn-split2b} have solutions we will use Banach fixed point theorem in some suitable complete metric spaces.

In $\cX$ we define a metric by
\begin{equation}\label{def:metric:X}
   d(\phi,\psi)
   = \sup\set{\dfrac{\|\phi(s,\xi) - \psi(s,\xi)\|}{\|\xi\|} : s \in \R_0^+
         \text{ and } \xi \in E_s\setminus \set{0}}.
\end{equation}
for each $\phi, \psi \in \cX$.
It is easy to see that $\cX$ is a complete metric space with the metric
defined by~\eqref{def:metric:X}.

Let $\cB_s$ be the space of functions $x \colon [s,+\infty[ \times E_s \to X$ such that
\begin{align}
   & x(t,0) =0 \text{ for every } t \ge s, \label{cond-x_m-0}\\
   & x(t,\xi) \in E_t \text{ for every } t \ge s
      \text{ and every } \xi \in E_s, \label{cond-x_m-0a}\\
   & \norms{x} = \sup\set{\dfrac{\|x(t,\xi)\|}{a(t,s) \|\xi\|}
      \colon t \ge s, \ \xi \in E_s \setm{0}} < +\infty. \label{cond-x_m-1}
\end{align}
From~\eqref{cond-x_m-1} we obtain the following estimate
\begin{equation} \label{cond-x_m-1a}
   \|x(t,\xi)\|
   \le a(t,s) \norms{x} \|\xi\|
\end{equation}
for every $t \ge s$ and every $\xi \in E_s$. It is easy to see that $\prts{\cB_s, \norms{\cdot}}$ is a Banach space.

\begin{lemma}\label{lemma:Exist-Suc-x_m}
   For each $\phi\in \cX$ and each $s \in \R_0^+$, there exists a unique function $x^{\phi}\in \cB_s$ satisfying equation~\eqref{eq:dyn-split2a}. Moreover
   \begin{align}
      & x^\phi(s,\xi) = \xi, \label{cond-x^phi_n(xi)=xi}\\
      & \norms{x^\phi} \le \dfrac{1}{1 - 2 \alpha},\label{eq:norm_x^phi}\\
      & \label{ineq:x^phi_m(xi)- x^phi_m(barxi)<=...}
         \|x^\phi(t,\xi) - x^\phi(t,\bar\xi)\|
         \le \dfrac{1}{1-2\alpha} a(t,s) \|\xi-\bar\xi\|
   \end{align}
   for every $t \ge s$ and $\xi$, $\bar\xi \in E_s$. Furthermore,
   \begin{equation} \label{ineq:d_n(x^phi-x^psi)<=d(phi,psi)}
      \norms{x^\phi - x^\psi}
      \le \dfrac{\alpha}{\prts{1-2\alpha}^2} d(\phi,\psi)
   \end{equation}
   for each $\phi, \psi \in \cX$.
\end{lemma}

\begin{proof}
   Given $\phi \in \cX$, we define an operator $J = J_\phi$ in~$\cB_s$ by
   \begin{equation} \label{def:J}
      (Jx)(t,\xi)=
      \begin{cases}
         \xi & \text{ if } t = s,\\
         T_{t,s} \xi + \dint_s^t T_{t,r} P_r
            f(r,x(r,\xi),\phi(r,x(r,\xi))) \dr & \text{ if } t > s.
      \end{cases}
   \end{equation}
   One can easily verify from~\eqref{cond-x_m-0}, \eqref{cond-phi-0} and
   \eqref{cond-f-0} that $(Jx)(t,0)=0$ for every $t \ge s$. It is also easy to see that $(Jx)(t, \xi) \in E_t$ for every $t \ge s$ and every $\xi \in E_s$ and thus $Jx$ verifies~\eqref{cond-x_m-0a}.

   Let $x \in \cB_s$ and let $\xi \in E_s$. From~\eqref{def:J},~\ref{eq:splitbounds1},~\eqref{cond-f-2}, \eqref{cond-phi-1a},~\eqref{cond-x_m-1a} and~\eqref{eq:tau} it follows for every $t > s$ that
   \begin{align*}
      \|(Jx)(t,\xi)\|
      & \le \| T_{t,s}P_s\| \, \|\xi \|
         + \dint_s^t \|T_{t,r} P_r\| \,
         \|f(r,x(r,\xi),\phi(r,x(r,\xi)))\| \dr\\
      & \le a(t,s) \|\xi\| + \dint_s^t a(t,r)  \, \Lip(f_r) \,
         \prts{\|x(r,\xi)\|+\|\phi(r,x(r,\xi))\|} \dr\\
      & \le a(t,s) \|\xi\| + \dint_s^t a(t,r)  \, \Lip(f_r) \,
         2 \|x(r,\xi)\| \dr\\
      & \le a(t,s) \|\xi\| + 2 \dint_s^t a(t,r)  \, \Lip(f_r) \,
         a(r,s) \norms{x} \|\xi\| \dr\\
      & \le a(t,s) \|\xi\| + 2 \alpha \norms{x} a(t,s) \|\xi\|\\
      & \le \prts{1 + 2 \alpha \norms{x}} a(t,s) \|\xi\|
   \end{align*}
   and this implies
   \begin{equation}\label{ineq:abs(Jx)}
      \norms{Jx} \le 1 + 2 \alpha \norms{x} <+\infty.
   \end{equation}
   Therefore we have the inclusion $J(\cB_s)\subset\cB_s$.

   We now show that $J$ is a contraction in $\cB_s$. Let $x,y\in \cB_s$. Then
   \begin{equation}\label{ineq:norm:J_x_m(xi)-J_y_m(bar_xi)}
      \begin{split}
         & \|(Jx)(t,\xi)-(Jy)(t,\xi)\|\\
         & \le \dint_s^t \|T_{t,r}P_r\| \
            \|f(r,x(r,\xi),\phi(r,x(r,\xi)))
          - f(r,y(r,\xi),\phi(r,y(r,\xi)))\| \dr
      \end{split}
   \end{equation}
   for every $t \ge s$ and every $\xi \in E_s$. By~\eqref{cond-f-1},~\eqref{cond-phi-1} and~\eqref{cond-x_m-1a} we have for every $r \ge s$
   \begin{equation}\label{ineq:norm:f_k(x_k)-f_k(y_k)}
      \begin{split}
         & \|f(r,x(r,\xi),\phi(r,x(r,\xi)))
            - f(r,y(r,\xi),\phi(r,y(r,\xi)))\|\\
         & \le \Lip(f_r) \prts{\| x(r,\xi) - y(r,\xi)\|
            + \|\phi(r,x(r,\xi)) - \phi(r,y(r,\xi))\|}\\
         & \le 2 \Lip(f_r) \| x(r,\xi) - y(r,\xi)\| \\
         & \le 2 \Lip(f_r) a(r,s) \|\xi\| \ \norms{x-y}.
      \end{split}
   \end{equation}
   Hence, from \eqref{ineq:norm:J_x_m(xi)-J_y_m(bar_xi)},~\ref{eq:splitbounds1}, \eqref{ineq:norm:f_k(x_k)-f_k(y_k)} and~\eqref{eq:tau} we have
   \begin{align*}
      \|(Jx)(t,\xi) -(Jy)(t,\xi)\|
      & \le 2 \|\xi\| \ \norms{x-y}
         \dint_s^t a(t,r) a(r,s) \Lip(f_r) \dr\\
      & \le 2 \alpha a(t,s) \|\xi\| \ \norms{x-y}
   \end{align*}
   for every $t \ge s$ and every $\xi \in E_s$ and this implies
      $$ \norms{Jx-Jy} \le 2 \alpha \norms{x-y}.$$
   By~\eqref{ine:alpha_beta} it follows that $\alpha < 1/2$ and therefore $J$ is a contraction in $\cB_s$. Because $\cB_s$ is a Banach space, by the Banach fixed point theorem, the map $J$ has a unique fixed point $x^\phi$ in~$\cB_s$, which is thus the desired function. Moreover, is obvious that~\eqref{cond-x^phi_n(xi)=xi} is true and by~\eqref{ineq:abs(Jx)} we have
      $$ \norms{x^\phi} \le 1 + 2 \alpha \norms{x^\phi}$$
   and since $\alpha < 1/2$ we have~\eqref{eq:norm_x^phi}.

   To prove~\eqref{ineq:x^phi_m(xi)- x^phi_m(barxi)<=...} we will first prove that, for every $x \in \cB_s$, if
   \begin{equation}\label{eq:|x(t,xi)-x(t,barxi)|}
      \|x(t,\xi)-x(t,\bar\xi)\|
      \le \dfrac{1}{1-2\alpha} a(t,s) \|\xi-\bar\xi\|
   \end{equation}
      $$ $$
   for every $t \ge s$ and every $\xi,\bar\xi \in E_s$, then
      $$ \|\prts{Jx}(t,\xi)-\prts{Jx}(t,\bar\xi)\|
         \le \dfrac{1}{1-2\alpha} a(t,s) \|\xi-\bar\xi\|$$
   for every $t \ge s$ and every $\xi,\bar\xi \in E_s$. In fact, by~\ref{eq:splitbounds1}, we have
   \begin{align*}
      \|\prts{Jx}(t,\xi)-\prts{Jx}(t,\bar\xi)\|
      & \le \|T_{t,s}P_s\| \|\xi-\bar\xi\|
            + \int_s^t \|T_{t,r} P_r\| \gamma(r) \dr \\
      & \le a(t,s) \|\xi-\bar\xi\| + \int_s^t a(t,r) \gamma(r) \dr,
   \end{align*}
   where $\gamma(r) = \|f(r,x(r,\xi),\phi(r,x(r,\xi))) - f(r,x(r,\bar\xi),\phi(r,x(r,\bar\xi)))\|.$ By~\eqref{cond-f-1},~\eqref{cond-phi-1} and~\eqref{eq:|x(t,xi)-x(t,barxi)|} we have
   \begin{align*}
      \gamma(r)
      & \le \Lip(f_r) \prts{\|x(r,\xi) - x(r,\bar\xi)\|
         + \|\phi(r,x(r,\xi))- \phi(r,x(r,\bar\xi))\|}\\
      & \le 2 \Lip(f_r) \|x(r,\xi) - x(r,\bar\xi)\| \\
      & \le \dfrac{2}{1-2\alpha} \Lip(f_r) a(r,s) \|\xi-\bar\xi\|
   \end{align*}
   and thus, by~\eqref{eq:tau},
   \begin{align*}
      & \|\prts{Jx}(t,\xi)-\prts{Jx}(t,\bar\xi)\|\\
      & \le a(t,s) \|\xi-\bar\xi\| + \dfrac{2}{1-2\alpha} \|\xi-\bar\xi\|
         \int_s^t a(t,r) a(r,s) \Lip(f_r) \dr\\
      & \le a(t,s) \|\xi-\bar\xi\|
         + \dfrac{2\alpha}{1-2\alpha} a(t,s) \|\xi-\bar\xi\|\\
      & = \dfrac{1}{1-2\alpha} a(t,s) \|\xi-\bar\xi\|.
   \end{align*}
   Now let $z \in B_s$ be given by $z(t,\xi)=T_{t,s}P_s \xi$ for every $t \ge s$ and $\xi \in E_s$. Since
      $$ \|z(t,\xi) - z(t,\bar\xi)\|
         \le a(t,s) \|\xi-\bar\xi\|
         \le \dfrac{1}{1-2\alpha} a(t,s) \|\xi-\bar\xi\|,$$
   we have
      $$ \|\prts{J^k z}(t,\xi) - \prts{J^k z}(t,\bar\xi)\|
         \le \dfrac{1}{1-2\alpha} a(t,s) \|\xi-\bar\xi\|$$
   for every $k \in \N$. Letting $k \to +\infty$ in the last inequality we have~\eqref{ineq:x^phi_m(xi)- x^phi_m(barxi)<=...}.

   Next we will prove~\eqref{ineq:d_n(x^phi-x^psi)<=d(phi,psi)}. Let $\phi, \psi \in \cX$. From~\eqref{eq:dyn-split2a} we have
   \begin{equation}\label{ineq:norm:x_m^phi(xi)-x_m^psi(xi)}
      \begin{split}
         & \|x^\phi(t,\xi) - x^\psi(t,\xi)\|\\
         & \le \int_s^t \|T_{t,r} P_r\| \,
            \|f(r,x^\phi(r,\xi),\phi(r,x^\phi(r,\xi)))
            - f(r,x^\psi(r,\xi),\psi(r,x^\psi(r,\xi)))\| \dr
      \end{split}
   \end{equation}
   for every $t \ge s$ and every $\xi \in E_s$. By~\eqref{cond-f-1},~\eqref{cond-phi-1},~\eqref{cond-x_m-1}, ~\eqref{def:metric:X},~\eqref{cond-x_m-1a} and~\eqref{eq:norm_x^phi} it follows that
   \begin{equation}\label{ineq:norm:f_k(x^phi_k...)- f_k(x^psi_k...)}
      \begin{split}
         & \|f(r,x^\phi(r,\xi),\phi(r,x^\phi(r,\xi)))
               - f(r,x^\psi(r,\xi),\psi(r,x^\psi(r,\xi)))\|\\
         & \le \Lip(f_r) \prts{\|x^\phi(r,\xi) - x^\psi(r,\xi)\|
            + \|\phi(r,x^\phi(r,\xi)) - \psi(r,x^\psi(r,\xi))\|}\\
         & \le \Lip(f_r) \prts{2 \|x^\phi(r,\xi) - x^\psi(r,\xi)\|
            + \|\phi(r,x^\psi(r,\xi)) - \psi(r,x^\psi(r,\xi))\|}\\
         & \le \Lip(f_r) \prtsr{2 a(r,s) \|\xi\| \norms{x^\phi-x^\psi}
            + \|x^\psi(r,\xi)\| d(\phi,\psi)}\\
         & \le \Lip(f_r) a(r,s) \|\xi\| \prtsr{2 \norms{x^\phi-x^\psi}
            + \dfrac{1}{1-2\alpha} d(\phi,\psi)}
      \end{split}
   \end{equation}
   for every $r \ge s$. Hence by~\eqref{ineq:norm:x_m^phi(xi)-x_m^psi(xi)}, \eqref{ineq:norm:f_k(x^phi_k...)- f_k(x^psi_k...)},~\ref{eq:splitbounds1} and~\eqref{eq:tau} we get
   \begin{align*}
      & \|x^\phi(t,\xi) - x^\psi(t,\xi)\|\\
      & \le \|\xi\| \prtsr{2 \norms{x^\phi-x^\psi}
         + \dfrac{1}{1-2\alpha} d(\phi,\psi)}
         \dint_s^t a(t,r)  a(r,s) \Lip(f_r) \dr\\
      & \le a(t,s) \|\xi\| \prtsr{2 \norms{x^\phi - x^\psi}
         + \dfrac{1}{1-2\alpha} d(\phi,\psi)} \alpha
   \end{align*}
   for every $t \ge s$ and every $\xi \in E_s$ and this implies
      $$ \norms{x^\phi-x^\psi}
         \le 2 \alpha \norms{x^\phi-x^\psi}
            + \dfrac{\alpha}{1-2\alpha} d(\phi,\psi).$$
   Therefore
      $$ \norms{x^\phi-x^\psi}
         \le \dfrac{\alpha}{\prts{1-2\alpha}^2} d(\phi,\psi),$$
   and we get~\eqref{ineq:d_n(x^phi-x^psi)<=d(phi,psi)}.
\end{proof}

Now we will turn our attention to identity~\eqref{eq:dyn-split2b}.

\begin{lemma} \label{lemma:equiv}
   Let $\phi \in \cX$. The following properties are equivalent:
   \begin{enumerate}[\lb=$\alph*)$,\lm=5mm]
      \item for every $s \in \R_0^+$, $t \ge s$ and $\xi \in E_s$ the identity~\eqref{eq:dyn-split2b} holds with $x = x^\phi$;
      \item for every $s \in \R_0^+$ and every $\xi \in E_s$
          \begin{equation} \label{eq:phi_n}
            \phi(s,\xi)
               = - \dint_s^{+\infty} (T_{r,s}|_{F_s})^{-1}
                  Q_r f(r,x^\phi(r,\xi), \phi(r,x^\phi(r,\xi))) \dr
          \end{equation}
          holds.
   \end{enumerate}
\end{lemma}

\begin{proof}
   First we prove that the integral in~\eqref{eq:phi_n} is convergent.
   From~\ref{eq:splitbounds2},~\eqref{cond-f-2},~\eqref{cond-phi-1a}, \eqref{cond-x_m-1a} and~\eqref{eq:lbd}, we conclude that for every $s \in \R_0^+$ and every $\xi \in E_s$
   \begin{align*}
      & \dint_s^{+\infty} \|(T_{r,s}|_{F_s})^{-1} Q(r)
         f(r,x^\phi(r,\xi),\phi(r,x^\phi(r,\xi)))\| \dr \\
      & \le \dint_s^{+\infty} \|(T_{r,s}|_{F_s})^{-1} Q(r)\| \
         \|f(r,x^\phi(r,\xi),\phi(r,x^\phi(r,\xi)))\| \dr \\
      & \le \dint_s^{+\infty} b(r,s) \Lip(f_r)
         \prts{\|x^\phi(r,\xi)\| + \|\phi(r,x^\phi(r,\xi))\|} \dr\\
      & \le 2 \|\xi\| \norms{x^\phi}
         \dint_s^{+\infty} b(r,s) a(r,s) \Lip(f_r) \dr\\
      & \le 2 \beta \|\xi\| \norms{x^\phi},
   \end{align*}
   and thus the integral converges.

   Now, let us suppose that~\eqref{eq:dyn-split2b} holds with $x = x^\phi$ for every $s \in \R_0^+$, every $t \ge s$ and every $\xi \in E_s$. Then, since $(T_{t,s}|_{F_s})^{-1}T_{t,r}|_{F_r}=(T_{r,s}|_{F_s})^{-1}$ for $s \le r \le t$, equation~\eqref{eq:dyn-split2b} can be written in the following equivalent form
   \begin{equation}\label{eq:equiv1A}
      \phi(s,\xi) = (T_{t,s}|_{F_s})^{-1} \phi(t,x^\phi(t,\xi)) -
      \dint_s^t (T_{r,s}|_{F_s})^{-1}
      Q_r f(r,x^\phi(r,\xi),\phi(r,x^\phi(r,\xi))) \dr.
   \end{equation}
   Using~\ref{eq:splitbounds2},~\eqref{cond-phi-1a} and~\eqref{cond-x_m-1a},
   we have
   \begin{align*}
      \|(T_{t,s}|_{F_s})^{-1} \phi(t,x^\phi(t,\xi))\|
      & = \|(T_{t,s}|_{F_s})^{-1} Q_t \, \phi(t,(x^\phi(t,\xi)))\| \\
      & \le b(t,s) \|x^\phi(t,\xi)\|\\
      & \le b(t,s) a(t,s) \|\xi\| \norms{x^\phi}
   \end{align*}
   and by~\eqref{eq:CondicaoTeo} this converges to zero when $t \to +\infty$.
   Hence, letting $t \to +\infty$ in \eqref{eq:equiv1A} we obtain the
   identity~\eqref{eq:phi_n} for every $s \in \R_0^+$ and every $\xi \in E_s$.

   We now assume that for every $s \in \R_0^+$, $t \ge s$ and $\xi \in E_s$ the identity~\eqref{eq:phi_n} holds. Therefore
      $$ T_{t,s} \phi(s,\xi)
         = - \dint_s^{+\infty} T_{t,s} (T_{r,s}|_{F_s})^{-1}
         Q(r) f(r,x^\phi(r,\xi),\phi(r,x^\phi(r,\xi))) \dr,$$ and
   thus it follows from \eqref{eq:phi_n} and the uniqueness of the sequences
   $x^\phi$ that
   \begin{align*}
      T_{t,s} \phi(s,\xi) & + \dint_s^t T_{t,r}
         Q_r f(r,x^\phi(r,\xi),\phi(r,x^\phi(r,\xi))) \dr\\
      & = - \dint_t^{+\infty} (T_{r,t}|_{F_r})^{-1}
            Q_r f(r,x^\phi(r,\xi),\phi(r,x^\phi(r,\xi))) \dr\\
      & = \phi(t,x^\phi(t,\xi))
   \end{align*}
   for every $s \in \R_0^+$, every $t \ge s$ and every $\xi \in E_s$. This proves the lemma.
\end{proof}

\begin{lemma} \label{lemma:Exist-Suc-phi}
   There is a unique $\phi \in \cX$ such that
      $$ \phi(s,\xi)
         = - \dint_s^{+\infty} (T_{r,s}|_{F_s})^{-1}
            Q_r f(r,x^\phi(r,\xi), \phi(r,x^\phi(r,\xi))) \dr$$
   for every $s \in \R_0^+$ and every $\xi \in E_s$.
\end{lemma}

\begin{proof}
   We consider the operator $\Phi$ defined for each $\phi\in \cX$ by
   \begin{equation} \label{eq:op-Phi}
      (\Phi \phi)(t,\xi) =
         - \dint_s^{+\infty} (T_{r,s}|_{F_s})^{-1} Q_r
         f(r,x^\phi(r,\xi), \phi(r,x^\phi(r,\xi))) \dr
   \end{equation}
   where $x^\phi \in \cB_s$ is the unique function given by Lemma~\ref{lemma:Exist-Suc-x_m}. It follows from~\eqref{cond-x_m-0}, \eqref{cond-phi-0}, \eqref{cond-f-0} and~\eqref{eq:op-Phi} that $(\Phi\phi)(s,0)=0$ for each $s\in\R_0^+$. It is easy to see that $(\Phi\phi)(t,\xi) \in F_t$ for every $(t, \xi) \in G$ and thus $\Phi\phi$ verifies~\eqref{cond-phi-0a}.

   Furthermore, given $s \in \R_0^+$ and $\xi, \bar\xi \in E_s$, by~\ref{eq:splitbounds2},~\eqref{cond-f-1},~\eqref{eq:lbd}, ~\eqref{ineq:x^phi_m(xi)- x^phi_m(barxi)<=...} and~\eqref{eq:lbd} we have
   \begin{align*}
      & \|(\Phi \phi)(s,\xi) - (\Phi \phi)(s,\bar\xi)\| \\
      & \le \dint_s^{+\infty} \|(T_{r,s}|_{F_s})^{-1} Q_r\|
         \cdot \|f(r,x^\phi(r,\xi),\phi(r,x^\phi(r,\xi)))
         - f(r,x^\phi(r,\bar\xi),\phi(r,x^\phi(r,\bar\xi)))\| \dr\\
      & \le \dint_s^{+\infty} b(r,s) \ \Lip(f_r) \
         2 \|x^\phi(r,\xi) - x^\phi(r,\bar\xi)\| \dr \\
      & \le \dfrac{2}{1-2\alpha} \|\xi-\bar\xi\|
         \dint_s^{+\infty} b(r,s) \ \Lip(f_r) \ a(r,s) \dr \\
      & \le \dfrac{2\beta}{1-2\alpha} \|\xi - \bar\xi\|
   \end{align*}
   Since $\alpha + \beta < 1/2$ we have
      $$ \|(\Phi \phi)(s,\xi) - (\Phi \phi)(s,\bar\xi)\|
         \le \|\xi - \bar\xi\|.$$
   Therefore $\Phi(\cX)\subset\cX$.

   We now show that $\Phi$ is a contraction in $\cX$. Given $\phi,\psi\in\cX$
   and $s \in \R_0^+$, let $x^{\phi}$ and $x^{\psi}$ be the unique sequences given
   by Lemma~\ref{lemma:Exist-Suc-x_m} respectively for $\phi$ and $\psi$.
   By~\ref{eq:splitbounds2},~\eqref{ineq:norm:f_k(x^phi_k...)- f_k(x^psi_k...)},~\eqref{ineq:d_n(x^phi-x^psi)<=d(phi,psi)}
   and~\eqref{eq:lbd} we have
   \begin{align*}
      & \|(\Phi \phi)(s,\xi)-(\Phi \psi)(s,\xi)\| \\
      & \le \dint_s^{+\infty} \|(T_{r,s}|_{F_s})^{-1} Q_r\|
         \|f(r,x^\phi(r,\xi),\phi(r,x^\phi(r,\xi)))
         - f(r,x^\psi(r,\xi),\psi(r,x^\psi(r,\xi)))\| \dr\\
      & \le \dint_s^{+\infty} b(r,s) \Lip(f_r) a(r,s) \|\xi\|
            \prtsr{2 \norms{x^\phi-x^\psi}
            + \dfrac{1}{1-2\alpha} d(\phi,\psi)} \dr\\
      & \le \dint_s^{+\infty} b(r,s) \Lip(f_r) a(r,s) \|\xi\|
            \prtsr{\dfrac{2\alpha}{\prts{1-2\alpha}^2} + \dfrac{1}{1-2\alpha}} d(\phi,\psi) \dr\\
      & \le \dfrac{1}{\prts{1-2\alpha}^2} \|\xi\| d(\phi,\psi)
         \dint_s^{+\infty} b(r,s) a(r,s) \Lip(f_r) \dr\\
      & \le \dfrac{\beta}{\prts{1-2\alpha}^2}
         \|\xi\| d(\phi,\psi)
   \end{align*}
   for every $s \in \R_0^+$ and every $\xi \in E_s$ and this implies
      $$ d(\Phi \phi, \Phi \psi)
         \le \dfrac{\beta}{\prts{1-2\alpha}^2}  d(\phi,\psi)$$
   Since $\dfrac{\beta}{\prts{1-2\alpha}^2} < 1$ it follows that $\Phi$ is a contraction in $\cX$. Therefore the map $\Phi$ has a unique fixed point $\phi$ in~$\cX$ that is the desired function.
\end{proof}

We are now in conditions to prove Theorem~\ref{thm:global}.

\begin{proof}[Proof of Theorem~$\ref{thm:global}$]
   By Lemma~\ref{lemma:Exist-Suc-x_m}, for each $\phi\in \cX$ there is
   a unique sequence $x^{\phi}\in\cB_s$ satisfying~\eqref{eq:dyn-split2a}. It
   remains to show that there is a $\phi$ and a corresponding $x^\phi$ that
   satisfies~\eqref{eq:dyn-split2b}. By Lemma~\ref{lemma:equiv}, this is
   equivalent to show that there is $\phi \in \cX$ and the corresponding $x^\phi \in \cB_s$ that satisfies~\eqref{eq:phi_n}. Finally, by
   Lemma~\ref{lemma:Exist-Suc-phi}, there is a unique solution of
   \eqref{eq:phi_n}. This establishes the existence of the invariant manifolds for $\delta>0$ sufficiently small. Moreover, for each $s \in \R_0^+$,
   $t \ge s$ and
   $\xi,\bar{\xi}\in E_s$ it follows from~\eqref{cond-phi-1} and~\eqref{ineq:x^phi_m(xi)- x^phi_m(barxi)<=...} that
   \begin{align*}
      & \|\Psi_{t-s}(s,\xi,\phi(s,\xi)) -
      \Psi_{t-s}(s,\bar\xi,\phi(s,\bar\xi))\| \\
      & \le \|x^\phi (t,\xi) - x^\phi (t,\bar\xi)\|
         + \|\phi(t,x^\phi (t,\xi)) - \phi(t,x^\phi (t,\bar\xi))\|\\
      & \le 2 \|x^\phi (t,\xi) - x^\phi (t,\bar\xi)\|\\
      & \le \dfrac{2}{1-2\alpha} a(t,s) \|\xi-\bar \xi\|.
   \end{align*}
   Hence we obtain~\eqref{eq:bound-thm1} and the theorem is proved.
\end{proof}
\section{Proof of Theorem~\ref{thm:local}}
We will now prove Theorem~\ref{thm:local}. Let $\tilde{f} \colon \R_0^+ \times X \to X$ the function defined by
   $$ \tilde{f}(t,x) =
      \begin{cases}
         f(r,x) & \text{if} \quad x \in B(R(r)) \\
         f\left(r, x R(r) / \|x\| \right) & \text{if} \quad x \notin B(R(r)).
      \end{cases}$$
Clearly, $\tilde f$ is a continuous function and, since $f_r|_{B(R(r))} \colon B(R(r)) \to X$ is a Lipschitz function for each $r \in \R_0^+$, it is easy to see that, for every $r \in \R_0^+$, the function $\tilde f_r \colon X \to X$, given by $\tilde{f}_r(x) = \tilde{f}(r,x)$, is
Lipschitz and  $\Lip(\tilde{f}_r) \le 2\Lip(f_r|_{B(R(r))})$. Thus, we have
   $$ \tilde{\alpha}
      = \sup_{(t,s) \in (\R_0^+)^2_>} \dfrac{1}{a(t,s)}
         \int_s^t a(t,r) a(r,s) \Lip(\tilde{f_r}) \dr
      \le 2 \alpha
      < + \infty,$$
   $$ \tilde{\beta}
      = \sup_{s \in \R_0^+}
         \dint_s^{+\infty} b(r,s) a(r,s) \Lip(\tilde{f_r}) \dr
         \le 2 \beta
         < + \infty$$
and
   $$ 2\tilde{\alpha} + \max\set{2\tilde{\beta}, \sqrt{\tilde{\beta}}}
      \le 4 \alpha + \max\set{4 \beta, \sqrt{2\beta}}
      < 1.$$
Hence, if $\tilde{\Psi}_\tau$ is the semiflow given by~\eqref{def:Psi} and corresponding to equation~\eqref{eq:ivp-nonli} with the perturbation $f$ replaced by $\tilde{f}$, then by Theorem~\ref{thm:global} we have that~\eqref{eq:invariance} holds for $\tilde{\Psi}_\tau$ and
\begin{equation}\label{eq:local:asy}
   \begin{split}
      \| \tilde\Psi_{t-s}(s,\xi,\phi(s,\xi))
      - \tilde\Psi_{t-s}(s,\bar\xi,\phi(s,\bar\xi))\|
      & \le \frac{2}{1-2\tilde\alpha} a(t,s) \|\xi - \bar\xi\|\\
      & \le \frac{2}{1-4 \alpha} a(t,s) \|\xi - \bar\xi\|\\
   \end{split}
\end{equation}
for every $(t,s) \in (R_0^+)^2_\ge$ and every $\xi, \bar\xi \in E_s$. In particular, if $\xi \in B(R(s)/(2S(s))) \cap E_s$ and $\bar\xi=0$, by~\eqref{eq:local:asy} we have
$(\xi,\phi(s,\xi)) \in B(R(s)/S(s))$ and from~\eqref{ine:sn-local} we have
   $$ \| \tilde\Psi_{t-s}(s,\xi,\phi(s,\xi))\|
      \le \frac{2}{1-4 \alpha} \, a(t,s) \|\xi\|
      < \frac{2}{1-4\alpha} \, a(t,s) \, \frac{R(s)}{S(s)}
      \le R(t)$$
and this implies
\begin{equation}\label{eq:local:inv}
   \tilde\Psi_\tau(\cV^*_{\phi,R/(2S)})
   \subseteq \cV^*_{\phi,R} \
   \text{for every } \tau \ge 0.
\end{equation}
where $\cV_\phi$ in~\eqref{eq:manifold-local} corresponds to the manifolds obtained by Theorem \ref{thm:global} for the perturbation $\tilde f$. Since $\tilde f_r|_{B(R(r))} = f_r|_{B(R(r))}$, from~\eqref{eq:local:inv} it follows~\eqref{thm:local:invar-local} and from~\eqref{eq:local:asy} we get~\eqref{thm:ineq:norm:F_mn(xi...)-F_mn(barxi...)-local}. This finishes the proof of Theorem~\ref{thm:local}.
\section*{Acknowledgments}
   This work was partially supported by FCT though Centro de Ma\-te\-m\'a\-ti\-ca da Universidade da Beira Interior (project PEst-OE/MAT/UI0212/2011).
\bibliographystyle{elsart-num-sort}
\def\cprime{$'$}

\end{document}